\documentclass[12pt]{article}
\usepackage{graphicx}
\usepackage{amsmath}
\usepackage{amssymb}
\usepackage{theorem}

\numberwithin{equation}{section}

\newtheorem{thm}{Theorem}[section]
\newtheorem{lemma}{Lemma}[section]
\newtheorem{prop}{Proposition}[section]
\newtheorem{cor}{Corollary}[section]
{\theorembodyfont{\rmfamily}

\newtheorem{examp}{Example}[section]

\newtheorem{rmk}{Remark}[section]
}

\newcommand{\qed}{\hfill \mbox{\raggedright \rule{.07in}{.1in}}}
 
\newenvironment{proof}{\vspace{1ex}\noindent{\bf
Proof}\hspace{0.5em}}{\hfill\qed\vspace{1ex}}
\newenvironment{pfof}[1]{\vspace{1ex}\noindent{\bf Proof of
#1}\hspace{0.5em}}{\hfill\qed\vspace{1ex}}

\newcommand{\R}{{\mathbb R}}

\newcommand{\Z}{{\mathbb Z}}
\newcommand{\D}{{\mathbb D}}

 \newcommand{\supp}{\operatorname{supp}}

\newcommand{\BIG}{\displaystyle}

\title{Mixing for invertible infinite measure systems}

\author{
Ian Melbourne \thanks{Mathematics Institute, University of Warwick, Coventry, CV4 7AL, UK}
}

 \date{24 November 2012; Revised 18 April 2014}

\begin{document}

\maketitle

 \begin{abstract}
In a recent paper, Melbourne and Terhesiu
[{Operator renewal theory and mixing rates for
  dynamical systems with infinite measure}, \emph{Invent.\ Math.\/} \textbf{189}
  (2012), 61--110] obtained results on mixing and mixing rates for a large
class of noninvertible maps preserving an infinite ergodic invariant measure.

Here, we are concerned with extending these results to the invertible setting.
Mixing is established for a large class of infinite measure invertible maps.
Assuming additional structure, in particular exponential contraction along stable manifolds, it is possible to obtain good results on mixing rates
and higher order asymptotics.
 \end{abstract}

\section{Introduction}

There is a well-developed theory of mixing and rates of mixing (decay of correlations) for large classes of dynamical systems with finite ergodic invariant measure.
Uniformly hyperbolic (Axiom~A) maps are topologically mixing up to a finite cycle~\cite{Smale67}, and topological mixing is equivalent to mixing for reasonable measures.    Then it is natural to restrict to the mixing case and focus on mixing rates.
The standard approach of Sinai, Bowen and Ruelle~\cite{Bowen75,Ruelle78,Sinai72} is to pass via a Markov partition to symbolic dynamics.
Quotienting out the stable directions leads to a uniformly expanding system
and the transfer operator $L$ restricted to H\"older functions contracts exponentially quickly to the invariant density.
This leads easily to exponential decay of correlations for the uniformly expanding
quotient system.
An elementary approximation argument then extends this result to
the original system.

A similar approach holds in the nonuniformly hyperbolic case for
systems modelled by a Young tower with exponential tails~\cite{Young98}.
(This incorporates classical examples such as dispersing billiards,
 H\'enon-like attractors, and Lorenz-like maps.)
The approach also extends to systems modelled by Young towers
with subexponential decay of correlations~\cite{Young99}.    In fact, there
was an oversight in the literature where many authors assumed that the approximation argument for passing from noninvertible to invertible had been checked, though Young's paper~\cite{Young99} was entirely in the noninvertible setting.
This was resolved recently by Gou\"ezel~\cite{GouezelPC} based on
ideas in~\cite{ChazottesGouezel12}; an exposition can be found in~\cite[Appendix~B]{MTapp2}.

For infinite measure systems, results on mixing and rates of mixing were recently obtained by
Melbourne and Terhesiu~\cite{MT12} in the noninvertible context.
A natural question is to extend these results to the invertible case via the above
approximation argument.   It turns out that this is significantly more complicated than in the finite measure case, for reasons that will become transparent shortly.  In this paper, we resolve the question of mixing in the invertible setting.  Also, we extend results on rates of mixing~\cite{MT12, Terhesiu-app, Terhesiusub} from the noninvertible case, but these extensions are satisfactory only when exponential (or very rapid) contraction is assumed along stable manifolds.  

We now discuss briefly the setting for the results in~\cite{MT12} and in this paper, and explain why the approximation step to pass from noninvertible to invertible is more difficult than in the finite measure case.
We will denote invertible maps by $f:M\to M$ and the corresponding
quotient maps by $\bar f:\bar X\to\bar X$.  The corresponding ergodic invariant measures are denoted $\mu$ and $\bar\mu$.

Let $L$ denote the (normalised) transfer operator for $\bar f$
(so $\int Lv\,w\,d\bar\mu=\int v\,w\circ \bar f\,d\bar\mu$ for $v\in L^1(\bar X)$,
$w\in L^\infty(\bar X)$).
In the finite measure case, we would expect that $L^nv\to\int v\,d\bar\mu$
and would study the rate of this convergence, which translates into the rate of decay of correlations.  However, in the infinite case $|L^nv|_1\to0$ for all
$v\in L^1$, so the aim is to find normalising constants $a_n$ such that
$a_nL^nv\to \int v\,d\bar\mu$ (in some sense) for suitable observables $v$.

An important ingredient for studying infinite measure systems, crucial 
also in~\cite{MT12}, is the existence of a ``good'' inducing set $\bar Y\subset\bar X$ of finite nonzero measure.
Define the {\em first return time} $\varphi:\bar Y\to \Z^+$  and the {\rm first return 
map} $\bar F=\bar f^\varphi:\bar Y\to\bar Y$ by setting
\begin{align} \label{eq-defn}
\varphi(y)=\inf\{n\ge1:\bar f^ny\in\bar Y\}, \quad \bar F(y)=\bar f^{\varphi(y)}(y).
\end{align}
It is assumed that $\bar f$ is conservative so that $\varphi$ is defined almost everywhere in $Y$, but that $\varphi$ is nonintegrable.

We require further that $\bar F:\bar Y\to\bar Y$ is uniformly expanding (in a sense that will be made precise later)
and that the tails of 
the first return time are regularly varying:
\begin{align} \label{eq-varphi}
\bar\mu(y\in \bar Y:\varphi(y)>n)=\ell(n)n^{-\beta},
\end{align}
where $\beta\in(0,1]$ and $\ell$ is a slowly varying function\footnote{A measurable function $\ell:(0,\infty)\to(0,\infty)$ is {\em slowly varying} if $\lim_{x\to\infty}\ell(\lambda x)/\ell(x)=1$ for all $\lambda>0$.  A {\em regularly varying} function is one of the form $\ell(x)x^q$ where $\ell(x)$ is slowly varying}.

Set $d_\beta= \frac{1}{\pi}\sin\beta\pi$ and define
\begin{align} \label{eq-an}
a_n = \begin{cases} d_\beta^{-1}\ell(n)n^{1-\beta}, & \beta\in(0,1) \\[.75ex]
\sum_{j<n}\ell(j)j^{-1}, & \beta=1 \end{cases}.
\end{align} 
Then for $\beta\in(\frac12,1]$, Melbourne and Terhesiu~\cite{MT12} extended results of Garsia and Lamperti~\cite{GarsiaLamperti62} from the scalar probability case to prove that
$\lim_{n\to\infty}|a_n1_{\bar Y}L^nv-\int v\,d\bar\mu|_\infty=0$ for H\"older
observables $v$ supported in $\bar Y$.
It follows that we obtain the mixing result
\[
\lim_{n\to\infty}a_n\int v\,w\circ \bar f^n\,d\bar\mu=\int v\,d\bar\mu\int w\,d\bar\mu,
\]
for all observables $v,w$ supported in $\bar Y$ with $v$ H\"older 
and $w$ integrable.
(The result is stated precisely in Theorem~\ref{thm-MTG}(a) below, and the extension to the invertible setting is Theorem~\ref{thm-main}(a).)

For $\beta\in(0,\frac12]$, it is known that such a result cannot hold without 
further assumptions.   Gou\"ezel~\cite{Gouezel11} used ideas of Doney~\cite{Doney97} to show that the result goes through under the additional 
{\em smooth tails} condition 
\begin{align} \label{eq-smooth1}
\bar\mu(y\in\bar Y:\varphi(y)=n)\le C\ell(n)n^{-(\beta+1)}.
\end{align}
(See Theorem~\ref{thm-MTG}(b) below, and Theorem~\ref{thm-main}(b) for
the extension to the invertible setting.)

In the first paragraph of the introduction, we mentioned that in the classical uniformly hyperbolic setting, there is an elementary approximation method that enables results on decay of correlations to be passed from a quotient noninvertible system to the underlying invertible one.
There is a single fundamental and transparent reason why the approximation method to pass from noninvertible to invertible maps is not straightforward in the infinite measure setting.
In the finite measure case, mixing and rates of mixing are proved for observables supported on the whole space.  In the infinite case, we work with observables supported in $\bar Y$ and $Y$.   Starting with observables $v,w:M\to\R$ supported in $Y$, approximation leads to
observables $\bar v$, $\bar w$ that are not supported in $\bar Y$.  Moreover, the support of $\bar v$ and $\bar w$ depends on the level of approximation.
This accounts for the delicate nature of some of the arguments in this paper.

\begin{rmk}  It turns out that all the problems above are associated with
approximating the $w$ observable.  If it is assumed that $w$ depends only on future coordinates so that it is necessary only to deal with the $v$ observable, then it is relatively easy to recover mixing and mixing rates from the noninvertible case regardless of rates of contraction along stable manifolds.  See Subsection~\ref{sec-v}.
\end{rmk}

\begin{examp}   \label{ex-weak}
An important class of noninvertible maps amenable to the methods in~\cite{MT12} are one-dimensional Pomeau-Manneville intermittent maps~\cite{PomeauManneville80}.  We focus here on the family of maps $\bar f:[0,1]\to[0,1]$ studied
by Liverani {\em et al.}~\cite{LiveraniSaussolVaienti99}.
For $\gamma>0$, these have the form 
\begin{align} \label{eq-LSV}
\bar f(x)=\begin{cases} x(1+2^\gamma x^\gamma), & x\in[0,\frac12) \\
2x-1, & x\in[\frac12,1] \end{cases},
\end{align}
and there is a unique (up to scaling) absolutely continuous $\bar f$-invariant
measure $\bar\mu$ for each $\gamma$.   The measure is finite if and only if
$\gamma<1$.   Taking $\bar Y=[\frac12,1]$,
conditions~\eqref{eq-varphi} and~\eqref{eq-smooth1}
are satisfied with $\beta=1/\gamma$ and $\ell$ asymptotically constant.
Moreover, these maps satisfy the remaining technical assumption
in~\cite{MT12} and here; namely that the first return map 
$\bar F:\bar Y\to\bar Y$ is a Gibbs-Markov map (see Section~\ref{sec-NUE} 
below).

The analogous class of invertible maps are studied in
Hu and Young~\cite{HuYoung95,Hu00}. These ``almost Anosov'' diffeomorphisms are uniformly hyperbolic except at one fixed point.  It is reasonable to expect that they often satisfy the required technical assumptions (it is easy to construct simplified  examples of intermittent diffeomorphisms which do so, see Section~\ref{sec-ex}) in which case our results on mixing apply.  Moreover, the examples in~\cite{HuYoung95} have one neutral expanding direction and one strictly contracting direction so that there is exponential contraction along stable manifolds, implying mixing rates similar to those in the noninvertible setting~\cite{MT12}.   However, Hu~\cite{Hu00} considers almost Anosov diffeomorphisms where both the contracting and expanding directions are neutral, for which our methods would
yield results on mixing but with poor mixing rates.
In independent work, Liverani and Terhesiu~\cite{LiveraniTerhesiusub} have developed a different technique which often yields optimal mixing rates in such situations, and moreover avoids assumptions of a Markov nature, but which is currently restricted to the case where there is a global smooth stable foliation.  See Remark~\ref{rmk-LT} for a more detailed comparison of the work presented here and in~\cite{LiveraniTerhesiusub}.
\end{examp}

The remainder of the paper is organised as follows.
In Section~\ref{sec-mixing}, we recall the results on mixing of~\cite{Gouezel11,MT12} from
the noninvertible setting, and state the corresponding result, Theorem~\ref{thm-main}, in the invertible setting.
In Section~\ref{sec-tower}, we collect some standard techniques related to Young towers.    
In Section~\ref{sec-key}, we prove a key estimate.
Sections~\ref{sec-a} and~\ref{sec-b} contain the proof of
Theorem~\ref{thm-main}.  In  Section~\ref{sec-rates}, we discuss 
rates of mixing.
Finally,  in Section~\ref{sec-ex}, we describe examples to which our results apply.

\paragraph{Notation}
We write
$a_n\ll b_n$ as $n\to\infty$ if there is a constant
$C>0$ such that $a_n\le Cb_n$ for all $n\ge1$.

\section{Statement of results on mixing}
\label{sec-mixing}

In this section, we describe the results of~\cite{Gouezel11,MT12} about mixing for nonuniformly expanding maps, and our main results about mixing for nonuniformly
hyperbolic diffeomorphisms.

\subsection{Mixing for nonuniformly expanding maps}
\label{sec-NUE}

The results of~\cite{Gouezel11,MT12} apply in particular
to systems with first return 
maps that are Gibbs-Markov (uniformly expanding plus good distortion).    This includes parabolic rational
maps of the complex plane (Aaronson {\em et al}~\cite{ADU93}) and Thaler's class of interval
maps with indifferent fixed points~\cite{Thaler95}.

We recall the key definitions~\cite{Aaronson,ADU93}.
Let $(\bar X,\bar\mu)$ be a Lebesgue space with 
countable measurable partition $\alpha_{\bar X}$.
Let $\bar f:\bar X\to \bar X$ be an ergodic, conservative, measure-preserving, Markov map transforming
each partition element bijectively onto a union of partition elements.
Recall that $\bar f$ is {\em topologically mixing} if for
all $a,b\in\alpha_{\bar X}$ there exists $N\ge1$ such that $b\subset \bar f^na$ for all $n\ge N$.

Let $\bar Y$ be a union of partition elements with $\bar\mu(\bar Y)\in(0,\infty)$.
Define the first return time $\varphi:\bar Y\to\R$ and first return map
$\bar F=\bar f^\varphi:\bar Y\to \bar Y$ as in~\eqref{eq-defn}.  Let $\alpha$ be the partition of $\bar Y$ consisting
of nonempty cylinders of the form $a\cap(\bigcap_{j=1}^{n-1} T^{-j}\xi_j)\cap T^{-n}\bar Y$ where $a,\xi_j\in\alpha_{\bar X}$, and $a\subset \bar Y$, $\xi_j\subset \bar X\setminus \bar Y$.
Fix $\theta\in(0,1)$ and define $d_\theta(x,y)=\theta^{s(x,y)}$
where the {\em separation time} $s(x,y)$ is the greatest integer $n\ge0$
such that $\bar F^nx$ and $\bar F^ny$ lie in the same partition element in $\alpha$.
It is assumed that the partition $\alpha$ separates orbits of $\bar F$, so
$s(x,y)$ is finite for all $x\neq y$; then $d_\theta$ is a metric.
Let $F_\theta(\bar Y)$ be the Banach space of $d_\theta$-Lipschitz
functions $v:\bar Y\to\R$ with norm $\|v\|_\theta=|v|_\infty+|v|_\theta$
where $|v|_\theta=\sup_{x\neq y}|v(x)-v(y)|/d_\theta(x,y)$. 

Define the potential function $p=\log\frac{d\bar\mu}{d\bar\mu\circ \bar F}:\bar Y\to\R$.
We require that $p$ is uniformly piecewise Lipschitz: that is,
$p|_a$ is $d_\theta$-Lipschitz for each $a\in\alpha$ and
the Lipschitz constants can be chosen independent of $a$.   
We also require the big images condition $\inf_a \bar\mu(\bar Fa)>0$.
Such a Markov map $\bar F:\bar X\to\bar X$ with partition $\alpha$, with uniformly piecewise Lipschitz potential
and satisfying the big images property, is called a {\em Gibbs-Markov map}.

Throughout we assume that $\bar\mu(\bar X)=\infty$ (in other words $\varphi$ is nonintegrable) and that $\bar\mu$ is normalised so that $\bar\mu(\bar Y)=1$.

\begin{thm}  \label{thm-MTG}
Let $\bar f:\bar X\to\bar X$ be as above with nonintegrable return time $\varphi:\bar Y\to\Z^+$ satisfying~\eqref{eq-varphi} and
Gibbs-Markov return map $\bar F=\bar f^\varphi:\bar Y\to\bar Y$.
Define $a_n$ as in~\eqref{eq-an}.
\begin{itemize}
\item[(a)] {\rm (Melbourne and Terhesiu~\cite{MT12})} If $\beta\in(\frac12,1]$, then
\[
\lim_{n\to\infty}
a_n\int_{\bar X}v\,w\circ \bar f^n\,d\bar\mu=\int_{\bar X} v\,d\bar\mu \int_{\bar X} w\,d\bar\mu,
\]
for all observables $v,w:\bar X\to\R$ supported in $\bar Y$
with $v\in F_\theta(\bar Y)$ and $w\in L^1$.
\item[(b)] {\rm (Gou\"ezel~\cite{Gouezel11})} The same conclusion holds also for $\beta\in(0,\frac12]$
under the additional ``smooth tails'' condition~\eqref{eq-smooth1}.
\end{itemize}
\end{thm}

\subsection{Mixing for nonuniformly hyperbolic diffeomorphisms}
\label{sec-NUH}

Let $f:M\to M$ be a diffeomorphism (possibly with singularities) defined on a
Riemannian manifold $(M,d)$.   
We assume that $f$ is nonuniformly hyperbolic in the sense of Young~\cite{Young98,Young99}.   The precise definitions are somewhat technical; here we are content to focus on the parts necessary for understanding this paper, referring to the original papers for further details.  

As part of this set up, there is an $f$-invariant ergodic conservative measure
$\mu$.  Let $Y\subset M$ with $\mu(Y)\in(0,\infty)$ and rescale so that $\mu(Y)=1$.
Define the first return time function
$\varphi:Y\to\Z^+$ given by $\varphi(y)=\inf\{n\ge1:f^ny\in Y\}$.
(By conservativity, $\varphi$ is finite almost everywhere on $Y$ and we may suppose without loss that $\varphi$ is finite on the whole of $Y$.)
Also, define the first return map $F=f^\varphi:Y\to Y$.
(In~\cite{Young98,Young99}, it is not required that $\varphi$ is the first return to $Y$, but this is a crucial ingredient in~\cite{MT12} and in this paper.)

Let $\{Z_j\}$ denote a measurable partition of $Y$ such that $\varphi$ is constant on partition elements.
Let $s$ denote the {\em separation time} with respect to the map $F:Y\to Y$.  
That is, if
$y,y'\in Y$, then $s(y,y')$ is the least integer $n\ge0$ such that $F^n y$, $F^ny'$ lie in distinct partition elements of $Y$.    

Let
$\{W^s\}$ and $\{W^u\}$ denote two measurable partitions of $Y$.
If $y\in Y$, the elements containing $y$ are labelled $W^s(y)$ and $W^u(y)$.

\begin{itemize}
\item[(P1)] 
$F(W^s(y))\subset W^s(Fy)$ for all $y\in Y$.
\item[(P2)]
There exist constants $C\ge1$, $\gamma_0\in(0,1)$ such that
\begin{itemize}
\item[(i)]  If $y'\in W^s(y)$, then $d(F^ny,F^ny')\le C\gamma_0^n$ for all $n\ge0$. 
\item[(ii)]  If $y'\in W^u(y)$, then $d(F^ny,F^ny')\le C\gamma_0^{s(y,y')-n}$ for all $0\le n\le s(y,y')$.
\item[(iii)] If $y,y'\in Y$, then $d(f^jy,f^jy')\le Cd(Fy,Fy')$
for all $0\le j<\min\{\varphi(y),\varphi(y')\}$.
\end{itemize}
\end{itemize}
Hypotheses (P1) and (P2) imply that $F:Y\to Y$ is uniformly hyperbolic with
stable and unstable disks $W^s$ and $W^u$.
We assume a local product structure, namely that each $W^s$ intersects each $W^u$ in precisely one point.

Let $\bar Y=Y/\sim$ where $y\sim y'$ if $y\in W^s(y')$ and
define the partition $\{\bar Z_j\}$ of $\bar Y$.
We obtain a well-defined first return time function  $\varphi:\bar Y\to\Z^+$ and
first return map $\bar F:\bar Y\to\bar Y$.
Define $\bar\pi_*\mu=\bar\mu$ where $\bar\pi:Y\to\bar Y$ is the quotient map.

\begin{itemize}
\item[(P3)]  The map $\bar F:\bar Y\to\bar Y$ and partition $\{\bar Z_j\}$ separate points in $\bar Y$.
(It follows that $d_\theta(y,y')=\theta^{s(y,y')}$ defines
a metric on $\bar Y$ for each $\theta\in(0,1)$.)
\item[(P4)]  $\bar F:\bar Y\to\bar Y$ is a Gibbs-Markov map with respect to the partition $\alpha=\{\bar Z_j\}$ with ergodic invariant probability measure $\bar\mu$.
\end{itemize}

We omit the additional assumptions in Young~\cite{Young98}
that guarantee that $\mu$ is a physical measure for $F:Y\to Y$.   
The results here do
not rely on this property.

Our main result is the following.

\begin{thm}  \label{thm-main}
Assume that $f:M\to M$ is a nonuniformly hyperbolic diffeomorphism satisfying
conditions (P1)--(P4) above.   Assume that the return time function
$\varphi:Y\to\Z^+$ is nonintegrable and has regularly varying tails
as in~\eqref{eq-varphi}.  

Assume further either (a) that $\beta\in(\frac12,1]$,
or (b) that $\beta\in(0,\frac12]$ and $\varphi$ has smooth tails as
in~\eqref{eq-smooth1}.

Then 
\[
\lim_{n\to\infty}
a_n\int_Mv\,w\circ f^n\,d\mu=\int_M v\,d\mu \int_M  w\,d\mu,
\]
for all H\"older observables $v,w:M\to\R$ supported in $Y$.
 \end{thm}

\begin{rmk}  The results in this paper are restricted to the situation where the first return quotient map $\bar F:\bar Y\to\bar Y$ is Gibbs-Markov.   However, the results in~\cite{MT12}, described in Theorem~\ref{thm-MTG}(a), hold in a much more general functional analytic setting which includes the important class of one-dimensional ``AFN'' maps studied in~\cite{Zweimuller00}.  (These are the non-Markovian analogue of  the maps studied by Thaler~\cite{Thaler95}.)
Results for invertible maps where the quotient map is AFN are obtained in~\cite{LiveraniTerhesiusub}.
\end{rmk}

\section{Towers and approximation of observables}
\label{sec-tower}

In this section, we recall standard material on modelling nonuniformly expanding maps and
nonuniformly hyperbolic diffeomorphisms by tower maps (one-sided and two-sided
respectively), and how to approximate observables to pass from two-sided to one-sided towers.

\subsection{One-sided towers}
\label{sec-tower1}

Let $\bar f:\bar X\to\bar X$ be a nonuniformly expanding map with Gibbs-Markov first return map $\bar F=\bar f^\varphi:\bar Y\to\bar Y$ as in
Subsection~\ref{sec-NUE}.
Define the {\em tower} $\bar \Delta=\{(y,j)\in \bar Y\times\Z:0\le j<\varphi(y)\}$
and the {\em tower map} $\bar f_{\bar\Delta}:\bar \Delta\to\bar \Delta$ given by
$\bar f_{\bar\Delta}(y,j)=(y,j+1)$ for $j\le \varphi(y)-2$
and $\bar f_{\bar\Delta}(y,j)=(\bar Fy,0)$ for $j= \varphi(y)-1$.

The base of the tower $\{(y,0):y\in \bar Y\}$ is naturally identified with
$\bar Y$ and so we may regard $\bar Y$ as a subset of both $\bar X$ and $\bar \Delta$.
Then $\mu_{\bar \Delta}=\bar\mu\times{\rm counting}$ is an $\bar f_{\bar\Delta}$-invariant measure on $\bar \Delta$.

Define the projection $\pi:\bar \Delta\to \bar X$, $\pi(y,j)=\bar f^jy$.
Then $\pi \bar f_{\bar\Delta} = \bar f \pi$ and $\pi_*\mu_{\bar \Delta}=\bar\mu$.
Thus $\bar f_{\bar\Delta}$ is an extension of $\bar f$ with the same first return map
$\bar F:\bar Y\to \bar Y$ and return time function $\varphi:\bar Y\to\Z^+$ as the original map.

\subsection{Two-sided towers}
\label{sec-tower2}

Now let $f:M\to M$ be a nonuniformly hyperbolic diffeomorphism, as in
Subsection~\ref{sec-NUH}, with 
first return map $F=f^\varphi:Y\to Y$ that quotients to
a Gibbs-Markov map $\bar F:\bar Y\to\bar Y$.  Form the quotient tower
$\bar\Delta$ as in Subsection~\ref{sec-tower1},

Similarly, starting from $F:Y\to Y$ and $\varphi:Y\to\Z^+$, we can form a tower $\Delta$
and tower map $f_\Delta:\Delta\to\Delta$ such that $F=f^\varphi:Y\to Y$ is the
first return map  for $f_\Delta$.   Again,
$\mu_\Delta=\mu\times{\rm counting}$ is an
$f_\Delta$-invariant measure on $\Delta$.   Define the semiconjugacy
$\pi:\Delta\to M$, $\pi(y,\ell)=f^\ell y$.  Then we can take
$\mu=\pi_*\mu_\Delta$ to be the measure on $M$ mentioned in Section~\ref{sec-NUH}.

We use the separation time for $F:Y\to Y$ to define a separation time on 
$\Delta$:
define $s((y,\ell),(y',\ell'))=s(y,y')$
if $\ell=\ell'$ and $0$ otherwise.
This drops down to separation times $s$ on $\bar\Delta$ and $\bar Y$.
(The separation time on $\bar Y$ coincides with the one defined in
Subsection~\ref{sec-NUE}.)

For $q=(y,\ell),q'=(y',\ell)\in\Delta$, we write $q'\in W^s(q)$ if $y'\in W^s(y)$ and
$q'\in W^u(q)$ if $y'\in W^u(y)$.
Conditions (P2) translate as follows.
\begin{itemize}
\item[(P2$'$)]
There exist constants $C\ge1$, $\gamma_0\in(0,1)$ such that
for all $q,q'\in \Delta$, $n\ge1$,
\begin{itemize}
\item[(i)]  If $q'\in W^s(q)$, then $d(\pi f_\Delta^nq,\pi f_\Delta^nq')\le C\gamma_0^{\psi_n(q)}$, and
\item[(ii)]  If $q'\in W^u(q)$, then $d(\pi f_\Delta^nq,\pi f_\Delta^nq')\le C\gamma_0^{s(q,q')-\psi_n(q)}$, 
\end{itemize}
\end{itemize}
where $\psi_n(q)=\#\{j=0,\dots,n-1:f^jq\in Y\}$ is the number of visits of $q$ to $Y$ by time $n$.

\begin{prop} \label{prop-W}
$d(\pi f_\Delta^n q,\pi f_\Delta^n q')\le C\gamma_0^{\min\{\psi_n(q),s(q,q')-\psi_n(q)\}}$ for all
$q,q'\in\Delta$, $n\ge1$.
\end{prop}

\begin{proof}  This is immediate from conditions (P2$'$) and the product 
structure on $Y$ (cf.~\cite[Corollary~5.3]{M07}).
\end{proof}

Given $\theta\in(0,1)$, we define the symbolic metric $d_\theta$ on
$\bar\Delta$ by setting $d_\theta(q,q')=\theta^{s(q,q')}$.
This restricts to the metric
$d_\theta$ defined on $\bar Y$ in Section~\ref{sec-NUE}.

\subsection{Approximation of observables}
Throughout, we work with H\"older observables $v_0,w_0:M\to\R$ with
fixed H\"older exponent $\eta\in(0,1]$.
Let  $v=v_0\circ\pi:\Delta\to\R$ be the lift of such a $C^\eta$ observable
$v_0:M\to\R$.   For each $k\ge1$, define $v_k:\Delta\to\R$,
\[
v_k(q)=\inf\{v(f_\Delta^kq'):s(q,q')\ge 2\psi_k(q)\}.
\]
We list some standard properties of $v_k$.
Let $L$ denote the transfer operator corresponding to $\bar f_{\bar\Delta}:\bar\Delta\to\bar\Delta$, defined by
$\int_{\bar\Delta}Lv\,w\,d\mu_{\bar\Delta}=
\int_{\bar\Delta}v\,w\circ\bar f_{\bar\Delta}\,d\mu_{\bar\Delta}$ for
$v\in L^1$, $w\in L^\infty$.

\begin{prop} \label{prop-tildev}
The function $v_k$ lies in $L^\infty(\Delta)$ and projects down to a Lipschitz 
observable $\bar v_k:\bar\Delta\to\R$.  Moreover, setting $\gamma=\gamma_0^\eta$
and $\theta=\gamma^{\frac12}$,
\begin{itemize}
\item[(a)] $|\bar v_k|_\infty=|v_k|_\infty\le |v_0|_\infty$.
\item[(b)]  
$|v\circ f_\Delta^k(q)-v_k(q)| \le C\|v_0\|_{C^\eta}\gamma^{\psi_k(q)}$ for $q\in\Delta$.
\item[(c)] $|\gamma^{\psi_k}|_1\to0$ as $k\to\infty$.
\item[(d)]  $\|L^k\bar v_k\|_\theta\le C\|v_0\|_{C^\eta}$.  
\end{itemize}
\end{prop}

\begin{proof}
Again, this is standard (see for example~\cite[Proposition~B.5]{MTapp2}).
If $s(q,q')\ge 2\psi_k(q)$, then $v_k(q)=v_k(q')$.
It follows that $v_k$ is piecewise constant on a measurable partition 
of $\Delta$, and hence is measurable, and that $\bar v_k$ is well-defined.
Part (a) is immediate.   

Recall that $v=v_0\circ\pi$ where $v_0:M\to\R$ is $C^\eta$.
Let $q\in\Delta$.
By Proposition~\ref{prop-W} and the definition of $v_k$,
\begin{align*}
|v\circ f_\Delta^k(q)-v_k(q)| &=|v_0(\pi f_\Delta^kq)-v_0(\pi f_\Delta^kq')|
 \le \|v_0\|_{C^\eta}d(\pi f_\Delta^kq,\pi f_\Delta^kq')^\eta
\\ & \le C^*\gamma^{\min\{\psi_k(q),s(q,q')-\psi_k(q)\}},
\end{align*}
where $q'$ is such that $s(q,q')\ge 2\psi_k(q)$.
In particular, $s(q,q')-\psi_k(q)\ge \psi_k(q)$, so we obtain part~(b).
Part~(c) is immediate by the bounded convergence theorem since
$\psi_k\to\infty$ as $k\to\infty$ almost everywhere.

The proof of part~(d) is a bit longer and we refer to~\cite[Proposition~B.5(c)]{MTapp2}.
\end{proof}

We require some further less standard properties that stem from the special role of $Y$ in infinite ergodic theory.

\begin{prop}   \label{prop-Y}  Let $k\ge1$.
\begin{itemize}
\item[(a)]  If $\supp v\subset Y$, then $\supp v_k\subset f_\Delta^{-k}Y$.
\item[(b)]  If $\supp v\subset Y$, then $\supp L^k\bar v_k\subset\bar Y$.
\end{itemize}
\end{prop}

\begin{proof}
Assume that $\supp v\subset Y$.
Suppose $q\not\in f_\Delta^{-k}(Y)$, that is $f_\Delta^kq\not\in Y$.
Then $f_\Delta^kq'\not\in Y$ for all $q'$ with $s(q,q')\ge k$,
and hence $v(f_\Delta^kq')=0$ for such points.  In particular, $v_k(q)=0$
proving (a).

Combining part~(a) with the first statement in Proposition~\ref{prop-tildev},
we obtain that $\supp\bar v_k\subset\bar f_{\bar\Delta}^{-k}\bar Y$.
Now $(L^k\bar v_k)(q)=\sum_{\bar f_\Delta^kq'=q}e^{p_k}(q')\bar v_k(q')$
where $p_k=p+p\circ f+\dots+p\circ f^{k-1}$.
Suppose $q\not\in \bar Y$ and that $\bar f_{\bar\Delta}^kq'=q$.
We have $\bar v_k(q')=v(f_\Delta^kq'')$ where $s(q'',q')\ge 2k$ and
hence $s(f_\Delta^kq'',q)\ge k$.  In particular, $f_\Delta^kq''\not\in Y$ so
$\bar v_k(q')=v(f_\Delta^kq'')=0$.   It follows that each term in 
$(L^k\bar v_k)(q)$ is zero, proving (b).
\end{proof}

\section{A key estimate}
\label{sec-key}

In this section, we prove the following key estimate.
Let $\bar Y_k=\bar f_{\bar\Delta}^{-k}\bar Y$.

\begin{lemma} \label{lem-key}
There exists a sequence $c_{k,n}$ with the property that
$\lim_{n\to\infty}c_{k,n}=0$ for all $k$ such that
\[
\Bigl|\int_{\bar\Delta} \Bigl(a_n L^{n-k}v-\int_{\bar\Delta} v\,d\bar\mu_{\bar\Delta}\Bigr) w\,d\bar\mu_{\bar\Delta}\Bigr|\le c_{k,n}\|v\||w|_\infty,
\]
for all $v\in F_\theta(\bar Y)$, $w\in L^\infty(\bar Y_k)$, $k,n\ge0$.
\end{lemma}

The proof of this result is entirely at the level of the quotient tower and we suppress the bars.  Also, we write $f$ and $\mu$, omitting the subscript $\Delta$.

We require some preliminary notation.   
Define $T_n,P:L^1(Y)\to L^1(Y)$ to be the operators
$T_n v=1_YL^n(1_Yv)$ and $Pv=1_Y\int_Y v\,d\mu$.    

Following~\cite{GouezelPhD} (see also~\cite{Gouezel05,M07,MTapp2} for example) we define a 
sequence of operators $A_n:L^1(Y)\to L^1(\Delta)$ that behave like $L^n$ but summing only over preimages that lie in $Y$ and do not return to $Y$ in $n$ iterates
of $f$.   This is analogous to $T_n=1_YL^n1_Y$ which is like $L^n$ but summing only over preimages that lie in $Y$ and have returned to $Y$ at time $n$.   
In symbols,
\[
(T_nv)(y)=\sum_{\substack{f^nz=y\\y,z\in Y}}e^{p_n}(z)v(z), \quad
(A_nv)(q)=
\!\!\!\!
\sum_{\substack{f^nz=q,z\in Y\\ f^jz\not\in Y, j=1,\dots,n}}
\!\!\!\!
e^{p_n}(z)v(z),
\]
and we make the convention that $T_0=A_0=I$.
Then $L^n1_Y$ is the convolution
\begin{align} \label{eq-convA}
L^n1_Y=\sum_{j=0}^nA_jT_{n-j}.
\end{align}
It is easy to check that 
$(A_jv)(y,\ell)=\delta_{j,\ell}v(y)$
where $\delta_{j,\ell}$ is the Kronecker delta.
In particular,
\begin{align} \label{eq-sumA}
1_\Delta=\sum_{j=0}^\infty A_j1_Y. 
\end{align}

\begin{prop} \label{prop-A}
$|1_{Y_k}A_j1_Y|_1\le \mu(\varphi>j)-\mu(\varphi>j+k)$
for all $j,k$,
\end{prop}

\begin{proof}
Since $1_{Y_k}A_j1_Y$ is supported on the $j$'th level of the tower,
$|1_{Y_k}A_j1_Y|_1= \mu_\Delta(\{\ell=j\}\cap Y_k)$.
But $\{\ell=j\}\cap Y_k\subset\{(y,\ell):\ell=j,\;j<\varphi(y)\le k+j\}$
and so
$\mu_\Delta(\{\ell=j\}\cap Y_k)\le\mu(\varphi>j)-\mu(\varphi>j+k)$.
\end{proof}

Using~\eqref{eq-convA} and~\eqref{eq-sumA}, we obtain
\begin{align*}
 & \int_{\Delta} \Bigl(a_n L^{n-k}v-\int_{\Delta} v\,d\mu\Bigr) w\,d\mu
  =\int_\Delta\Bigl( a_n\sum_{j=0}^{n-k}A_jT_{n-k-j}v-\Bigl(\int_\Delta v\,d\mu\Bigr)\sum_{j=0}^\infty A_j1_Y\Bigr) 1_{Y_k}w\,d\mu \\
& \qquad  =\int_\Delta w\sum_{j=0}^{n-k}1_{Y_k}A_j1_Y\Bigl(a_nT_{n-k-j}v-Pv\Bigr)\,d\mu-
(Pv)\int_\Delta w\sum_{j=n-k+1}^\infty 1_{Y_k}A_j1_Y\,d\mu.
\end{align*}

Define $U_n:F_\theta(Y)\to F_\theta(Y)$, $U_nv=a_nT_nv$.   
Then 
$\int_{\Delta} \Bigl(a_n L^{n-k}v-\int_{\Delta} v\,d\mu\Bigr) w\,d\mu
=E_1(k,n)+E_2(k,n)-E_3(k,n)$, where
\begin{align*}
E_1(k,n) & = 
\int_\Delta w \sum_{j=0}^{n-k}1_{Y_k}A_j1_Y\Bigl(U_{n-k-j}v-Pv\Bigr)\,d\mu \\
E_2(k,n) & = 
\int_\Delta w \sum_{j=0}^{n-k}1_{Y_k}A_j1_Y(a_n-a_{n-k-j})T_{n-k-j}v\,d\mu 
 \\ & \qquad \qquad =\int_\Delta w \sum_{j=0}^{n-k}1_{Y_k}A_j1_Y\Bigl(\frac{a_n}{a_{n-k-j}}-1\Bigr)U_{n-k-j}v\,d\mu \\
E_3(k,n) & =
(Pv)\int_\Delta w\sum_{j=n-k+1}^\infty 1_{Y_k}A_j1_Y\,d\mu.
\end{align*}
We make use of the fact, proved in~\cite{MT12}, that 
\begin{align} \label{eq-Un}
\lim_{n\to\infty}\|U_n-P\|=0.
\end{align}

To estimate $E_1(k,n)$, it is useful to recall
the following ``discrete dominated convergence theorem''.

\begin{prop} \label{prop-DCT}
Let $b_{n,j}\in\R$, $n,j\ge0$.  Suppose that $\lim_{n\to\infty}b_{n.j}=0$
for each $j$ and there is a sequence $c_j>0$ such that
$|b_{n,j}|\le c_j$ for all $j,n$ and $\sum_{j=0}^\infty c_j<\infty$.
Then $\lim_{n\to\infty} \sum_{j=0}^\infty b_{n,j}=0$.
\qed
\end{prop}

\begin{prop} \label{prop-E1}
$\lim_{n\to\infty}E_1(k,n)=0$ for each fixed $k$.
\end{prop}

\begin{proof}
We keep $k$ fixed throughout.  By Proposition~\ref{prop-A},  
\[
|E_1(k,n)|\le  |w|_\infty\sum_{j=0}^{n-k}|1_{Y_k}A_j1_Y|_1|U_{n-k-j}v-Pv|_\infty \le
\|v\||w|_\infty\sum_{j=0}^{n-k}b_{n,j}
\]
where
$b_{n,j}=(\mu(\varphi>j)-\mu(\varphi>j+k))
\|U_{n-k-j}-P\|$.
By~\eqref{eq-Un}, $\lim_{n\to\infty}b_{n,j}=0$ for each $j$,
and there is a constant $C$ such that
$b _{n,j}\le c_j=C(\mu(\varphi>j)-\mu(\varphi>j+k))$.
Moreover, $\sum_{j=0}^\infty c_j=C\sum_{j=0}^{k-1}\mu(\varphi>j)<\infty$.
Now apply Proposition~\ref{prop-DCT}.
\end{proof}

Next, it follows from~\eqref{eq-Un} that $\|U_n\|$ is bounded, so
\begin{align*}
|E_2(k,n)| & 
\le |w|_\infty\sum_{j=0}^{n-k} |1_{Y_k}A_j1_Y|_1|U_{n-k-j}v|_\infty\Bigl|\frac{a_n}{a_{n-k-j}}-1\Bigr|
\\ &
\ll \|v\||w|_\infty \sum_{j=0}^{n-k}|1_{Y_k}A_j1_Y|_1\Bigl|\frac{a_n}{a_{n-k-j}}-1\Bigr|.
\end{align*}
It is convenient to split the sum into two parts, defining
\begin{align*}
E_2'(k,n) & = \sum_{0\le j\le n/2}|1_{Y_k}A_j1_Y|_1\Bigl|\frac{a_n}{a_{n-k-j}}-1\Bigr|, \\
E_2''(k,n)   & = \sum_{n/2<j\le n-k}|1_{Y_k}A_j1_Y|_1\Bigl|\frac{a_n}{a_{n-k-j}}-1\Bigr|.
\end{align*}

\begin{prop} \label{prop-E2'}
$\lim_{n\to\infty}E_2'(k,n)=0$ for each fixed $k$.
\end{prop}

\begin{proof}
We keep $k$ fixed throughout. 
By Proposition~\ref{prop-A},
$E_2'(k,n)\le \sum_{0\le j\le n/2}b_{n,j}$ where
$b_{n,j}=(\mu(\varphi>j)-\mu(\varphi>j+k))|\frac{a_n}{a_{n-k-j}}-1|$.
Recalling the definition in~\eqref{eq-an}, we note that $a_n$ is regularly varying (this is immediate for $\beta<1$ and is a consequence of
Karamata for $\beta=1$).
By Potter's bounds, $a_n/a_{n-k-j}$ is bounded for $j,n$ with $j\le n/2$.
Hence $b_{n,j}\ll c_j=\mu(\varphi>j)-\mu(\varphi>j+k)$ which is summable
as in Proposition~\ref{prop-E1}.
Also $\lim_{n\to\infty}a_n/a_{n-k-j}=1$ for each fixed $j$.
Hence we can again apply Proposition~\ref{prop-DCT}.
\end{proof}

Define  $\tilde\ell(n)=\sum_{j<n}\ell(j)j^{-1}$.
\begin{prop} \label{prop-E2''}
$E_2''(k,n)\ll \begin{cases} k\ell(n)^2n^{-(2\beta-1)} & \beta<1 \\
k\tilde\ell(n)\ell(n)n^{-1} & \beta=1 \end{cases}$ for all $k,n$.
\end{prop}

\begin{proof}
We have $a_n/a_{n-k-j}\ll a_n$, so by Proposition~\ref{prop-A},
\begin{align*}
|E_2''(k,n)| & \ll a_n\!\!\!\!\sum_{n/2<j\le n-k}\!\!\!\! (\mu(\varphi>j)-\mu(\varphi>j+k))
\le a_nk\mu(\varphi>n/2)
\ll a_nk\mu(\varphi>n),
\end{align*}
and the result follows by definition of $a_n$.
\end{proof}

\begin{prop} \label{prop-E3}  For $k\le n/2$,
$|E_3(k,n)| \ll \Bigl|\int_\Delta v\,d\mu\Bigr||w|_\infty k\ell(n)n^{-\beta}$.
\end{prop}

\begin{proof}
By Proposition~\ref{prop-A}, suppressing the factor 
$|\int_{\Delta}v\,d\mu||w|_\infty$,
\begin{align*}
|E_3(k,n)| & \le 
\sum_{j=n-k+1}^\infty 
|1_{Y_k}A_j1_Y|_1
\le 
\sum_{j=n-k+1}^\infty 
\mu(\varphi>j)-\mu(\varphi>j+k) \\ &
=\sum_{j=n-k+1}^n\mu(\varphi>j)\le k\mu(\varphi>n-k+1).
\end{align*}
Since $\mu(\varphi>n)$ is regularly varying, it follows from Potter's bounds that $\mu(\varphi>n-k+1)/\mu(\varphi>n)$ is bounded for $k,n$ with $k\le n/2$.
Hence $|E_3(k,n)|\ll k\mu(\varphi>n)$ proving the result.
\end{proof}

\begin{pfof}{Lemma~\ref{lem-key}}
Since $\int_{\Delta} \Bigl(a_n L^{n-k}v-\int_{\Delta} v\,d\mu\Bigr) w\,d\mu
=E_1(k,n)+E_2(k,n)-E_3(k,n)$,
the result follows from Propositions~\ref{prop-E1}--\ref{prop-E3}.
\end{pfof}

\section{Proof of Theorem~\ref{thm-main}(a)}
\label{sec-a}

In this section, we prove part (a) of our main result, Theorem~\ref{thm-main},
establishing mixing for infinite measure nonuniformly hyperbolic diffeomorphisms with $\beta>\frac12$.

Given $C^\eta$ observables $v_0,w_0:M\to\R$, supported in $Y$, let
$v=v_0\circ \pi,w=w_0\circ \pi:\Delta\to\R$ be the lifted observables to the tower.  (Since we are interested in observables on $Y$ and $F=f^\varphi$ is a first return map, $\pi|_Y={\rm id}_Y$ and we could simply write $v_0=v$ and $w_0=w$, but this does not seem helpful.)
Since $\pi:\Delta\to M$ is a semiconjugacy and $\mu=\pi_*\mu_\Delta$,
to prove Theorem~\ref{thm-main} it is suffices to work 
entirely in the context of the tower map 
$f_\Delta:\Delta\to\Delta$ and the one-sided quotient
$\bar f_{\bar\Delta}:\bar\Delta\to\bar\Delta$.
For notational simplicity, from now on we write $f$, $\bar f$, $\mu$ and $\bar\mu$ instead of $f_\Delta$, $\bar f_{\bar\Delta}$,
$\mu_\Delta$ and $\bar\mu_{\bar\Delta}$.
Thus the result we want to prove becomes
\[
\lim_{n\to\infty}
a_n\int_\Delta v\,w\circ f^n\,d\mu=\int_\Delta v\,d\mu \int_\Delta  w\,d\mu,
\]
for $v,w:Y\to\R$ H\"older.
Again for notational simplicity, from now on we write $\|v\|$
instead of $\|v\|_{C^\eta}$ and similarly for $w$.

Since the arguments for dealing with $v$ and $w$ are very different, it is convenient to give the arguments separately in Subsections~\ref{sec-v}
and~\ref{sec-w}.   The combined argument is given in Subsection~\ref{sec-vw}.

\subsection{The $v$ observable}
\label{sec-v}

In this subsection, we show how to deal with the $v$ observable:
we assume the set up of Theorem~\ref{thm-main}(a), but we
suppose that $w$ depends only on future coordinates (that is, $w$ is constant along stable disks and hence projects to an observable on $\bar\Delta$).

\begin{lemma} \label{lem-v}
There exists a constant $C>0$ such that
\[
\Bigl|
a_n\int_{\Delta}v\,w\circ f^n\,d\mu-\int_\Delta v\,d\mu \int_\Delta w\,d\mu
\Bigr|\le C\|U_n-P\|\|v\||w|_1.
\]
for all $v,w:Y\to\R$ where $v$ is H\"older and 
$w$ is $L^\infty$ such that $w$ depends only on future coordinates, and all $n\ge0$.
\end{lemma}

\begin{proof}
Write
\begin{align*}
& a_n\int_\Delta v\, w\circ f^n\,d\mu -\int_\Delta v\,d\mu\int_\Delta w\,d\mu
\\ & \qquad =a_n\int_\Delta v\circ f^\ell\, w\circ f^{n+\ell}\,d\mu -\int_\Delta v\circ f^\ell\,d\mu\int_\Delta w\,d\mu 
= I_1(\ell,n)+I_2(\ell,n)-I_3(\ell,n),
\end{align*}
where
\begin{align*}
& I_1(\ell,n)  = a_n\int_\Delta v_\ell\, w\circ f^{n+\ell}\,d\mu-\int_\Delta v_\ell\,d\mu\int_\Delta w\,d\mu, \\
& I_2(\ell,n)  = a_n\int_\Delta (v\circ f^\ell-v_\ell)\,w\circ f^{n+\ell}\,d\mu, \quad
 I_3(\ell,n)  = \int_\Delta (v\circ f^\ell-v_\ell)\,d\mu\int_\Delta w\,d\mu.
\end{align*}

By Proposition~\ref{prop-tildev} and the assumption on $w$, both 
$v_\ell$ and $w$ depend only on future coordinates.  Hence we can write
\begin{align*}
\int_\Delta v_\ell\, w\circ f^{n+\ell}\,d\mu
 & = \int_{\bar\Delta} \bar v_\ell\, w \circ \bar f^{n+\ell}\,d\bar\mu  
  = \int_{\bar \Delta} L^nL^\ell \bar v_\ell\, w\,d\bar\mu
\end{align*}
Moreover, 
$\supp L^\ell\bar v_\ell\subset \bar Y$ 
by Proposition~\ref{prop-Y}(b) and 
$\supp w\subset \bar Y$ by assumption, so $L^nL^\ell \bar v_\ell\, w=T_nL^\ell \bar v_\ell\, w$.  Hence
\begin{align*}
I_1(\ell,n)  &  = a_n\int_{\bar \Delta} T_nL^\ell \bar v_\ell\, w\,d\bar\mu
-\int_{\bar \Delta} L^\ell \bar v_\ell\,d\bar\mu\,\int_{\bar \Delta} w\,d\bar\mu 
 = \int_{\bar \Delta} \Bigl(a_nT_nL^\ell \bar v_\ell-P
L^\ell \bar v_\ell\,\Bigr)w\,d\bar\mu,
\end{align*}
and so
\begin{align*}
|I_1(\ell,n)| & \le  |(a_nT_n-P)L^\ell\bar v_\ell|_\infty\,|w|_1  
 \le \|U_n-P\|\|L^\ell\bar v_\ell\|_\theta|w|_1  
 \ll \|U_n-P\|\|v\||w|_1,
\end{align*}
where we have used Proposition~\ref{prop-tildev}(d).

Next, by Proposition~\ref{prop-tildev}(b),
\[
|I_2(\ell,n)|\ll a_n\|v\||w|_\infty|\gamma^{\psi_\ell}|_1,  \quad |I_3(\ell,n)| \ll \|v\||w|_1|\gamma^{\psi_\ell}|_1.
\]
By Proposition~\ref{prop-tildev}(c), $\lim_{\ell\to\infty}I_j(\ell,n)=0$ for all $n$ for $j=2,3$.  The result follows.
\end{proof}

\subsection{The $w$ observable}
\label{sec-w}

In this subsection, we show how to deal with the $w$ observable:
we assume the set up of Theorem~\ref{thm-main}(a), but we
suppose that $v$ depends only on future coordinates.

\begin{lemma} \label{lem-w}
Let $c_{k,n}$ be as in Lemma~\ref{lem-key}.  There exists a constant $C>0$ such that
\[
\Bigl|
a_n\int_{\Delta}v\,w\circ f^n\,d\mu-\int_\Delta v\,d\mu \int_\Delta w\,d\mu\Bigr|\le C(c_{k,n}+|\gamma^{\psi_k}|_1)\|v\|\|w\|,
\]
for all $v,w:Y\to\R$ H\"older 
such that $v$ depends only on future coordinates, and all $0\le k\le n$.
\end{lemma}

\begin{proof}
Write
\begin{align} \label{eq-I123}
 & a_n\int_{\Delta}v\,w\circ f^n\,d\mu-\int_\Delta v\,d\mu \int_\Delta w\,d\mu  \\
&\quad  =  a_n\int_{\Delta}v\circ f^k\,w\circ f^{n+k}\,d\mu-\int_\Delta v\,d\mu \int_\Delta w\circ f^k\,d\mu 
 = 
I_1(k,n)+I_2(k,n)-I_3(k,n), \nonumber
\end{align}
where
\begin{align*}
I_1(k,n) & = a_n\int_{\Delta}v\circ f^k\,w_k\circ f^n\,d\mu-\int_\Delta v\,d\mu \int_\Delta w_k\,d\mu, \\
 I_2(k,n) & = a_n\int_{\Delta}v\circ f^k\,(w\circ f^k-w_k)\circ f^n\,d\mu, \\
I_3(k,n) & =\int_\Delta v\,d\mu \int_\Delta (w\circ f^k-w_k)\,d\mu.
\end{align*}

By the same argument used for $I_1(\ell,n)$ in Subsection~\ref{sec-v},
$\int_{\Delta} v\circ f^k\, w_k\circ f^n\,d\mu=
\int_{\bar\Delta} v\circ \bar f^k\, \bar w_k\circ \bar f^n\,d\bar\mu$.  Hence
for all $k\le n$,
\[I_1(k,n)  =
a_n\int_{\bar\Delta} v\circ \bar f^k\, \bar w_k\circ \bar f^n\,d\bar\mu-\int_{\bar\Delta} v\,d\bar\mu\int_{\bar\Delta} \bar w_k\,d\bar\mu 
=\int_{\bar\Delta} \Bigl(a_n L^{n-k}v-\int_{\bar\Delta} v\,d\bar\mu\Bigr) \bar w_k\,d\bar\mu.
\]
By Lemma~\ref{lem-key} and
Proposition~\ref{prop-tildev}(a),
$|I_1(k,n)|\le c_{k,n}\|v\||w|_\infty$.

Next, note that 
\begin{align*}
I_2(k,n) & =a_n\int_\Delta (1_Yv)\circ f^k [1_{f^{-k}Y}(w\circ f^k-w_k)]\circ f^n\,d\mu
\\ & =
a_n\int_\Delta (1_Y\, 1_Y\circ f^n)\circ f^k\,   v\circ f^k\, (w\circ f^k-w_k)\circ f^n\,d\mu,
\end{align*}
so by Proposition~\ref{prop-tildev}(b), for $n\ge k$,
\begin{align*}
|I_2(k,n)| & \ll  
|v|_\infty\|w\|a_n\int_\Delta (1_Y\, 1_Y\circ f^n)\circ f^k \,\gamma^{\psi_k}\circ f^n\,d\mu   
\\ & =|v|_\infty\|w\|a_n\int_{\bar\Delta} 1_{\bar Y}\, (1_{\bar Y}\circ \bar f^k\, \gamma^{\psi_k})\circ \bar f^{n-k}\,d\bar\mu   
\\ & = |v|_\infty\|w\|a_n\int_{\bar\Delta} (L^{n-k}1_{\bar Y})\, 1_{\bar Y_k} \,\gamma^{\psi_k}\,d\bar\mu
=|v|_\infty\|w\|(I_2'(k,n)+I_2''(k,n)),
\end{align*}
where
\begin{align*}
I_2'(k,n) & = \int_{\bar\Delta} \Bigl(a_nL^{n-k}1_{\bar Y}-\int_{\bar\Delta}1_{\bar Y}\,d\bar\mu\Bigr)\, 1_{\bar Y_k} \,\gamma^{\psi_k}\,d\bar\mu, \quad
 I_2''(k,n)  = \int_{\bar\Delta} 1_{\bar Y_k} \,\gamma^{\psi_k}\,d\bar\mu.
\end{align*}
It follows from Lemma~\ref{lem-key} that
$|I_2'(k,n)|\le c_{k,n}\|1_{\bar Y}\||1_{Y_k}\gamma^{\psi_k}|_\infty
\le c_{k,n}$.
Also $|I_2''(k,n)|\le |\gamma^{\psi_k}|_1$.

Finally, by Proposition~\ref{prop-tildev}(b),
$|I_3(k,n)|\ll |v|_1\|w\||\gamma^{\psi_k}|_1$.
\end{proof}

\subsection{The $v$ and $w$ observables}
\label{sec-vw}

Finally, we consider the general case where $v$ and $w$ depend on the past and future.

\begin{lemma} \label{lem-vw}
Let $c_{k,n}$ be as in Lemma~\ref{lem-key}.  There exists a constant $C>0$ such that
\[
\Bigl|
a_n\int_{\Delta}v\,w\circ f^n\,d\mu-\int_\Delta v\,d\mu \int_\Delta w\,d\mu\Bigr|\le C(c_{k,n}+|\gamma^{\psi_k}|_1)\|v\|\|w\|,
\]
for all $v,w:Y\to\R$ H\"older, and all $0\le k\le n$.
\end{lemma}

\begin{proof}
We combine the methods used to prove Lemmas~\ref{lem-v} and~\ref{lem-w}.
Write
\begin{align*}
& a_n\int_\Delta v\,w\circ f^n\,d\mu-\int_\Delta v\,d\mu\int_\Delta w\,d\mu 
\\ &\qquad  = 
a_n\int_\Delta v\circ f^{\ell+k}\,w\circ f^{k+n+\ell}\,d\mu-\int_\Delta v\circ f^\ell\,d\mu\int_\Delta w\circ f^k\,d\mu 
\\ & \qquad = I_1(k,\ell,n)+I_2(k,\ell,n)+I_3(k,\ell,n)-I_4(k,\ell,n)-I_5(k,\ell,n),
\end{align*}
where
\begin{align*}
I_1(k,\ell,n) & = a_n\int_\Delta v_\ell\circ f^k\,w_k\circ f^{n+\ell}\,d\mu-\int_\Delta v_\ell\,d\mu\int_\Delta w_k\,d\mu, && \\
I_2(k,\ell,n)  & = a_n\int_\Delta v_\ell\circ f^k\,(w\circ f^k-w_k)\circ f^{n+\ell}\,d\mu,  \\
I_3(k,\ell,n) & = a_n\int_\Delta (v\circ f^\ell-v_\ell)\circ f^k\,w\circ f^{k+n+\ell}\,d\mu, \\
 I_4(k,\ell,n)  & = \int_\Delta v_\ell\,d\mu \int_\Delta (w\circ f^k-w_k)\,d\mu, \\
I_5(k,\ell,n) & = \int_\Delta (v\circ f^\ell-v_\ell)\,d\mu \int_\Delta w\circ f^k\,d\mu.
\end{align*}

Now
\begin{align*}
\int_\Delta v_\ell\circ f^k\,w_k\circ f^{n+\ell}\,d\mu & =
\int_{\bar\Delta}\bar v_\ell\,\bar w_k\circ \bar f^{n-k+\ell}\,d\bar\mu 
= \int_{\bar\Delta}L^{n-k}L^\ell \bar v_\ell\,\bar w_k\,d\bar\mu.
\end{align*}
Hence
\begin{align*}
I_1(k,\ell,n) & = a_n\int_{\bar\Delta} L^{n-k}L^\ell \bar v_\ell\,\bar w_k\,d\bar\mu
- \int_{\bar\Delta}L^\ell \bar v_\ell\,d\bar\mu \int_{\bar\Delta}\bar w_k\,d\bar\mu
 \\ & 
=\int_{\bar\Delta}\Bigl(a_nL^{n-k}L^\ell \bar v_\ell- \int_{\bar\Delta}L^\ell \bar v_\ell\,d\bar\mu\Bigr)\bar w_k\,d\bar\mu.
\end{align*}
By Proposition~\ref{prop-Y}(b), $\supp L^\ell\bar v_\ell\subset \bar Y$.
Hence by Lemma~\ref{lem-key} and Proposition~\ref{prop-Y}(a,d),
\[
|I_1(k,\ell,n)|\le c_{k,n}\|L^\ell\bar v_\ell\||\bar w_k|_\infty
\ll c_{k,n}\|v\||w|_\infty.
\]

Arguing as for $I_2(k,n)$ in the proof of Lemma~\ref{lem-w},
\begin{align*}
 |I_2(k,\ell,n)| 
 & \ll |v|_\infty \|w\| a_n\int_\Delta 1_Y\circ f^{k+\ell}\,1_Y\circ f^{k+n+\ell}\,\gamma^{\psi_k}\circ f^{n+\ell}\,d\mu \\ &
= |v|_\infty \|w\| a_n\int_{\bar\Delta} 1_{\bar Y} \,(1_{\bar Y}\circ \bar f^{k}\,\gamma^{\psi_k})\circ \bar f^{n-k}\,d\bar\mu \\ & 
= |v|_\infty \|w\| a_n\int_{\bar\Delta} (L^{n-k}1_{\bar Y}) \,1_{\bar Y_k}\,\gamma^{\psi_k}\,d\bar\mu \\ & 
= |v|_\infty \|w\|\Bigl\{ \int_{\bar\Delta} (a_nL^{n-k}1_{\bar Y}-\int_{\bar\Delta}1_{\bar Y}\,d\bar\mu) \,1_{\bar Y_k}\,\gamma^{\psi_k}\,d\bar\mu+\int_{\bar\Delta}1_{\bar Y_k}\,\gamma^{\psi_k}\,d\bar\mu\Bigr\} \\&
\le  |v|_\infty \|w\|(c_{k,n}+|\gamma^{\psi_k}|_1).
\end{align*}

By Proposition~\ref{prop-tildev}(b),
\begin{align*}
|I_3(k,\ell,n)| & \ll a_n\|v\||w|_\infty|\gamma^{\psi_\ell}|_1, \quad
|I_4(k,\ell,n)|\ll |v|_\infty\|w\||\gamma^{\psi_k}|_1, 
 \\ & 
|I_5(k,\ell,n)|\ll \|v\| |w|_1|\gamma^{\psi_\ell}|_1.
\end{align*}

Combining these estimates, letting $\ell\to\infty$, and using
Proposition~\ref{prop-tildev}(c), we obtain the required result.
\end{proof}

\begin{pfof}{Theorem~\ref{thm-main}(a)}
Recall that $\lim_{n\to\infty}c_{k,n}=0$ for each $k$.  Hence
it follows from Lemma~\ref{lem-vw} 
that $\limsup_{n\to\infty}|a_n\int_\Delta v\,w\circ f^n\,d\mu-\int_\Delta  v\,d\mu \int_\Delta   w\,d\mu|\ll |\gamma^{\psi_k}|_1$.
The result follows from Proposition~\ref{prop-tildev}(c).
\end{pfof}

\section{Proof of Theorem~\ref{thm-main}(b)}
\label{sec-b}

In this section, we prove part~(b) of our main theorem, dealing with the
case $\beta\in(0,\frac12]$.   Recall that it is necessary to assume an additional condition even for the noninvertible result, namely the
smooth tails condition $\mu(\varphi=n)\le C\ell(n) n^{-(\beta+1)}$.
The proof of Theorem~\ref{thm-main}(b) is almost identical to the proof
of part (a).  The only estimate that requires modification is the one
for $E_2''(k,n)$ in Proposition~\ref{prop-E2''}.

In fact, the milder smooth tail condition
\begin{align} \label{eq-smooth}
\mu(\varphi=n)\le C\ell(n)n^{-(\beta+q)},
\end{align}
suffices, provided that $q>1-\frac12\beta$.

\begin{rmk}  \label{rmk-smooth}
Let $q\in[0,1]$.
Given that $\mu(\varphi>n)=\ell(n)n^{-\beta}$, condition~\eqref{eq-smooth} is equivalent to the condition that
\begin{align*}
|\ell(n+1)-\ell(n)|\ll \ell(n) n^{-q}.
\end{align*}
\end{rmk}

Recall that $a_n=d_\beta^{-1}\ell(n)n^{1-\beta}$.  We list some elementary consequences of smooth tails.

\begin{prop} \label{prop-tails}
Assume smooth tails with $q\in(1-\beta,1]$.  Then
\begin{itemize}
\item[(i)] $\mu(\varphi>j)-\mu(\varphi>j+k)\ll k\ell(j)j^{-(\beta+q)}$.
\item[(ii)] $\BIG\Bigl|\frac{a_{n+j}}{a_n}-1\Bigr|\ll jn^{-q}$.
\qed
\end{itemize}
\end{prop}

\begin{prop}   \label{prop-E2''_tails}
Assume the smooth tails condition with $q\in(1-\beta,1)$.
Then $E_2''(k,n)\ll k\ell(n) n^{-(\beta+2q-2)}$.
If $q=1$, then
$E_2''(k,n)\ll k\ell(n)n^{-\beta}\log n$.
\end{prop}

\begin{proof}
Recall that
$E_2''(k,n)  = \sum_{n/2<j\le n-k}|1_{Y_k}A_j1_Y|_1\Bigl|\frac{a_n}{a_{n-k-j}}-1\Bigr|$.
By Propositions~\ref{prop-A} and~\ref{prop-tails},
\begin{align*}
E_2''(k,n) & \ll  k\sum_{j=n/2}^{n-k}\ell(j)j^{-(\beta+q)}(k+j)(n-k-j)^{-q}
 \ll k\ell(n) n^{1-\beta-q}\sum_{j=n/2}^{n-k}(n-k-j)^{-q} 
\end{align*}
and the result follows.
 \end{proof}

In particular, if $q>1-\frac12\beta$, then
$\lim_{k\to\infty}\lim_{n\to\infty}E_2''(k,n)=0$.
All the other terms were covered
in the proof of Theorem~\ref{thm-main}(a) and hence the proof of
Theorem~\ref{thm-main}(b) is complete.

\section{Mixing rates}
\label{sec-rates}

As mentioned in the introduction, we obtain some results on mixing rates and higher order asymptotics, but in general they are fairly weak.
However, if we assume 
sufficiently rapid contraction along stable manifolds and smooth tails with $q$ large enough, then we obtain essentially optimal results.   For simplicity, throughout this section we assume exponential contraction along stable manifolds
and the smooth tails condition~\eqref{eq-smooth1} with $q=1$.

\subsection{Some more estimates}

\begin{lemma} \label{lem-moreE1}
Assume the smooth tails condition~\eqref{eq-smooth1}.  
If the mixing rate in the noninvertible case is given by
$\|a_nT_n-P\|\ll (\log n)^c n^{-\tau}$, where $\tau\in(0,1)$, $c\ge0$, then
$|E_1(k,n)|\ll \|v\||w|_\infty k (\log n)^c  n^{-\tau}$ for all $k,n$.
\end{lemma}

\begin{proof}
Recall that
$|E_1(k,n)|  \le 
\|v\||w|_\infty  \sum_{j=0}^{n-k}|1_{Y_k}A_j1_Y|_1\|U_{n-k-j}-P\|$.
By Propositions~\ref{prop-A} and~\ref{prop-tails}(i),
\[
|E_1(k,n)|  \le 
\|v\||w|_\infty  k\sum_{j=0}^{n-k}\ell(j)j^{-(\beta+1)}(\log (n-k-j))^c(n-k-j)^{-\tau},
\]
which gives the desired result since the convolution of $\ell(n)n^{-(\beta+1)}$
and $(\log n)^cn^{-\tau}$ is $O((\log n)^cn^{-\tau})$.
\end{proof}

\begin{lemma} \label{lem-moreE2}
Assume the smooth tails condition~\eqref{eq-smooth1}.  Then
$|E_2(k,n)|\ll \|v\||w|_\infty (k^2n^{-1}+k\ell(n)n^{-\beta}\log n)$
for $k\le n/3$.
\end{lemma}

\begin{proof}
We already obtained this estimate for $E_2''(k,n)$ in Proposition~\ref{prop-E2''_tails} so it remains to handle $E_2'(k,n)$.
The following very rough estimate suffices.

Recall that
$E_2'(k,n)  = \sum_{0\le j\le n/2}|1_{Y_k}A_j1_Y|_1\Bigl|\frac{a_n}{a_{n-k-j}}-1\Bigr|$.
By Propositions~\ref{prop-A} and~\ref{prop-tails},
\begin{align*}
|E_2'(k,n)| & \ll 
k\sum_{j=1}^{n/2} \ell(j)j^{-(\beta+1)}(k+j)(n-k-j)^{-1}
\ll kn^{-1}\sum_{j=1}^{n/2} \ell(j)j^{-(\beta+1)}(k+j)
\\ & = k^2n^{-1}\sum_{j=1}^{n/2} \ell(j)j^{-(\beta+1)}
+ kn^{-1}\sum_{j=1}^{n/2} \ell(j)j^{-\beta}
\ll k^2n^{-1}+k\ell(n)n^{-\beta}.
\end{align*}
\end{proof}

\begin{cor} \label{cor-rate}  
Assume the smooth tails condition~\eqref{eq-smooth1}.  
If the mixing rate in the noninvertible case is given by
$\|a_nT_n-P\|\ll (\log n)^c n^{-\tau}$, where $\tau\in(0,1)$, $c\ge0$, then
\begin{align*}
& \Bigl|a_n\int_Mv\,w\circ f^n\,d\mu-\int_M v\,d\mu \int_M  w\,d\mu\Bigr|
\\ & \qquad\qquad \ll \|v\|\|w\|\{k(\log n)^c n^{-\tau}+k^2n^{-1}+k\ell(n)n^{-\beta}\log n+
|\gamma^{\psi_k}|_1\},
\end{align*}
for all H\"older $v,w:M\to\R$ supported in $Y$, and all $0\le k\le n/3$.
\end{cor}

\begin{proof}  This combines the various estimates obtained throughout this
paper.    We claim that these combine to give the estimate
\[
\|v\|\|w\|\{k(\log n)^c n^{-\tau}+k^2n^{-1}+k\ell(n)n^{-\beta}\log n+
|\gamma^{\psi_k}|_1
+a_n|\gamma^{\psi_\ell}|_1\}.
\]
Since $\ell$ is arbitrary the result follows from Proposition~\ref{prop-tildev}(c).   

To prove the claim it is required to estimate the five terms $I_j(k,\ell,n)$, $j=1,\dots,5$, that appear in Section~\ref{sec-vw}.
By Proposition~\ref{prop-tildev}(b), $|I_3(k,\ell,n)|\ll a_n|\gamma^{\psi_\ell}|_1$, 
$|I_4(k,\ell,n)|\ll |\gamma^{\psi_k}|_1$, 
$|I_5(k,\ell,n)|\ll |\gamma^{\psi_\ell}|_1$.

The calculations in Section~\ref{sec-vw} show that
\begin{align*}
|I_1(k,\ell,n)| & 
\ll c_{k,n}, \quad
|I_2(k,\ell,n)|  \ll (c_{k,n}+|\psi_k|_1),
\end{align*}
where $c_{k,n}$ is as in Lemma~\ref{lem-key}.
Moreover, by the proof of Lemma~\ref{lem-key}, we can obtain an estimate
for $c_{k,n}$ by estimating the expressions 
$E_j(k,\ell,n)$, $j=1,2,3$.  This is done using Lemma~\ref{lem-moreE1}, Lemma~\ref{lem-moreE2} and Proposition~\ref{prop-E3} respectively.
\end{proof}

\subsection{Exponential contraction}
\label{sec-exp}

Suppose that there is exponential contraction along stable manifolds.
This corresponds to strengthening condition (P2)(i) in
Section~\ref{sec-NUH}.  
\begin{description}
\item[(E)]  
If $y'\in W^s(y)$, then $d(f^ny,f^ny')\le C\gamma_0^n$ for all $n\ge1$.
\end{description}
In this case, $\gamma^{\psi_k}$ is replaced by $\gamma^k$.

\begin{thm} \label{thm-exp}
Assume that $f:M\to M$ is a nonuniformly hyperbolic diffeomorphism satisfying
conditions (P1)--(P4) and the exponential contraction condition
(E).   
Suppose that $\varphi$ satisfies conditions~\eqref{eq-varphi}
and~\eqref{eq-smooth1}.

If the mixing rate in the noninvertible case is given by
$\|a_nT_n-P\|\ll (\log n)^c n^{-\tau}$, where $\tau\in(0,1)$, $c\ge0$, then
\[
\Bigl|a_n\int_Mv\,w\circ f^n\,d\mu-\int_M v\,d\mu \int_M  w\,d\mu\Bigr|
\ll \|v\|\|w\|\{(\log n)^{c+1} n^{-\tau}+\ell(n)n^{-\beta}(\log n)^2\},
\]
for all $v,w:Y\to\R$ H\"older, and all $n\ge0$.
\end{thm}

\begin{proof}
Take $k=q\log n$ with $q$ large in Corollary~\ref{cor-rate}
\end{proof}

The following refinement of Theorem~\ref{thm-exp} is useful for passing
results on higher order asymptotics in the noninvertible context 
over to invertible systems.  

\begin{thm} \label{thm-exp2}
Assume that $f:M\to M$ is a nonuniformly hyperbolic diffeomorphism satisfying
conditions (P1)--(P4) and the exponential contraction condition
(E).   
Suppose that $\varphi$ satisfies conditions~\eqref{eq-varphi}
and~\eqref{eq-smooth1}.

Let $b_n$ be a bounded sequence such that $b_{n+1}-b_n\ll n^{-1}$.
If the mixing rate in the noninvertible case is given by
$\|a_nT_n-b_nP\|\ll (\log n)^c n^{-\tau}$, where $\tau\in(0,1)$, $c\ge0$, 
then
\[
\Bigl|a_n\int_Mv\,w\circ f^n\,d\mu-b_n\int_M v\,d\mu \int_M  w\,d\mu\Bigr|
\ll \|v\|\|w\|\{(\log n)^{c+1} n^{-\tau}+\ell(n)n^{-\beta}(\log n)^2\},
\]
for all $v,w:Y\to\R$ H\"older, and all $n\ge0$.
\end{thm}

\begin{proof}
In the notation of Lemma~\ref{lem-vw}, we obtain
$a_n\int_\Delta v\,w\circ f^n\,d\mu-b_n\int_\Delta v\,d\mu \int_\Delta  w\,d\mu=
J_1(k,\ell,n)+
I_2(k,\ell,n)+
I_3(k,\ell,n)-
b_nI_4(k,\ell,n)-
b_nI_5(k,\ell,n)$,
where
\begin{align*}
J_1(k,\ell,n)  & = a_n\int_\Delta v_\ell\circ f^k\,w_k\circ f^{n+\ell}\,d\mu-b_n\int_\Delta v_\ell\,d\mu\int_\Delta w_k\,d\mu \\
& 
=\int_{\bar\Delta}\Bigl(a_nL^{n-k}L^\ell \bar v_\ell-b_n \int_{\bar\Delta}L^\ell \bar v_\ell\,d\bar\mu\Bigr)\bar w_k\,d\bar\mu.
\end{align*}
Since $b_n$ is bounded, it suffices to verify that $J_1(k,\ell,n)$ is bounded by the terms listed in Corollary~\ref{cor-rate}.

As in the proof of Lemma~\ref{lem-key}, we reduce to considering three terms
 $F_j(k,n)$, $j=1,2,3$, given by $F_2(k,n)=E_2(k,n)$, $F_3(k,n)=b_nF_3(k,n)$ and
\begin{align*}
F_1(k,n)  & = 
\int_\Delta w \sum_{j=0}^{n-k}1_{Y_k}A_j1_Y\Bigl(U_{n-k-j}v-b_nPv\Bigr)\,d\mu=
F_1'(k,n)+F_1''(k,n)+F_1'''(k,n),
\end{align*}
where
\begin{align*}
F_1'(k,n) & = \int_\Delta w \sum_{j=0}^{n-k}1_{Y_k}A_j1_Y\Bigl(U_{n-k-j}v-b_{n-k-j}Pv\Bigr)\,d\mu, \\
F_1''(k,n) & = \int_\Delta w \sum_{0\le j\le n/2}1_{Y_k}A_j1_Y(b_{n-k-j}-b_n)Pv\,d\mu, \\
F_1'''(k,n) & = \int_\Delta w \sum_{n/2<j\le n-k}1_{Y_k}A_j1_Y(b_{n-k-j}-b_n)Pv\,d\mu.
\end{align*}
The argument for $E_1(k,n)$ in Lemma~\ref{lem-moreE1} goes through unchanged for $F_1'(k,n)$.
Also, by Propositions~\ref{prop-A} and~\ref{prop-tails}(i),
\[
|F_1'''(k,n)|\ll |v|_1|w|_1\sum_{n/2<j\le n-k}|1_{Y_k}A_j1_Y|_1
\ll k\mu(\varphi>n/2)
\ll k\ell(n)n^{-\beta}.
\]

So far, we used only that $b_n$ is bounded.  Finally, using Propositions~\ref{prop-A} and~\ref{prop-tails}(i) and the extra condition on $b_n$, for $k\le n/3$,
\begin{align*}
|F_1''(k,n)| & \le |v|_1|w|_1\sum_{0\le j\le n/2}|1_{Y_k}A_j1_Y|_1|b_{n-k-j}-b_n|  \\ & \ll kn^{-1}\sum_{0\le j\le n/2} \ell(j)j^{-(\beta+1)}(k+j) 
\ll k^2n^{-1}+k\ell(n)n^{-\beta}.
\end{align*}
This completes the proof.
\end{proof}

\begin{examp}
For the maps~\eqref{eq-LSV} studied in~\cite{LiveraniSaussolVaienti99}, conditions~\eqref{eq-varphi} 
and~\eqref{eq-smooth1} are satisfied
and $\ell(n)$ is asymptotically constant.
For $\beta\in(\frac12,1)$,
Melbourne and Terhesiu~\cite{MT12} obtained a mixing rate in the noninvertible case of the form 
\begin{align} \label{eq-mixnoninv}
\|T_n-\sum_{j\ge1} d_ja_n^{-j}P\|\ll c_n,
\end{align}
where $c_n=n^{-\frac12}$,
$d_1=1$ and the remaining constants
$d_j$ are typically nonzero.
For such expressions we make the convention that the sum is finite, disregarding the terms for which $a_n^{-j}\ll c_n$.

if we assume exponential contraction along stable manifolds then setting
$c=0$, $\tau=\beta-\frac12$, and $b_n=\sum_{j\ge1} d_ja_n^{-(j-1)}$ in
Theorem~\ref{thm-exp2}, we obtain
the almost identical mixing rate 
\begin{align} \label{eq-mixinv}
\Bigl|\int_Mv\,w\circ f^n\,d\mu-\sum_{j\ge1} d_ja_n^{-j}\int_M v\,d\mu \int_M  w\,d\mu\Bigr|
\ll \|v\|\|w\|c_n',
\end{align}
where $c_n'=(\log n)n^{-\frac12}$.
As in~\cite{MT12}, this is optimal for $\beta\ge \frac34$ and we obtain second order asymptotics for $\beta>\frac34$.

Using more refined methods, Terhesiu~\cite{Terhesiu-app,Terhesiusub} has obtained much stronger results for various classes of systems based on the properties of $\varphi$.   All of these results apply in particular to the maps~\eqref{eq-LSV}.
In~\cite{Terhesiu-app}, it is still assumed that $\beta>\frac12$, but with the
improved mixing rate $c_n=n^{-\beta}$ in~\eqref{eq-mixnoninv},
so for invertible systems with exponential contraction, we obtain~\eqref{eq-mixinv} with $c_n'=(\log n)n^{-\beta}$.
As in~\cite{Terhesiu-app}, this is optimal for $\beta\ge \frac23$ and we obtain second order asymptotics for $\beta>\frac23$.

Finally, Terhesiu~\cite{Terhesiusub} drops the restriction $\beta>\frac12$ and obtains mixing rates for all $\beta\in(0,1)$.   The assumptions on $\varphi$ are yet more restrictive but still include the maps~\eqref{eq-LSV} and yield 
$c_n=(\log n)n^{-1}$ in~\eqref{eq-mixnoninv}.  Accordingly,
for invertible systems with exponential contraction, we obtain $c_n'=(\log n)^2n^{-1}$ in~\eqref{eq-mixinv}.
As in~\cite{Terhesiu-app}, this is optimal for $\beta\ge \frac12$ and we obtain second order asymptotics for $\beta>\frac12$.
\end{examp}

\subsection{Subexponential contraction}

It is possible to relax hypothesis (E)
to allow uniform but subexponential decay along stable manifolds~\cite{AlvesAzevedosub}.
In the stretched exponential case $d(f^ny,f^ny')\le C\gamma_0^{n^\delta}$, our results on mixing rates and higher order asymptotics are essentially unchanged (the factor $(\log n)^{c+1}$ is changed to $(\log n)^{c+1/\delta}$).
More generally, we obtain essentially optimal mixing rates and higher order asymptotics if $d(f^ny,f^ny')\le Cn^{-q}$ with $q$ sufficiently large.

\subsection{Slow contraction}

The original assumption (P2)(i) assumes
that the contraction is as slow as the expansion. That is, expansion or
contraction occurs only at returns to the set $Y$.   This is natural for many examples, such as billiards, where there are long periods of time during which there is no contraction or expansion.

In such situations we define $\psi=1_Y$ and note that $\psi_k=\sum_{j=0}^{k-1}\psi\circ f^j$.
By the methods in~\cite{GouezelPC,MTapp2}, it can be shown that
$|\gamma^{\psi_k}|_1\ll \mu(\varphi>k)$.   Combined with smooth tails conditions, this can be used to obtain estimates on mixing rates.
The results are too poor to write down here.  (We note that a slight simplification is possible since the $k^2n^{-1}$ term in Lemma~\ref{lem-moreE2} can be removed with some extra effort.)

\section{Examples}
\label{sec-ex}

In this section, we describe some examples in which our hypotheses can be verified.
All of our results apply to Examples~\ref{ex-trivial}--\ref{ex-SW}.
In Example~\ref{Ex-E}, Theorem~\ref{thm-main} holds
but hypothesis (E) fails.

\begin{examp}
\label{ex-trivial}
We begin with a particularly elementary example.  
Define $f:X\to X$, $X=[0,1]^2$, by setting
\[
f(x_1,x_2)=\begin{cases} (x_1(1+2^\gamma x_1^\gamma)\,,\, \frac12 x_2) & 
0\le x_1<\frac12 \\
(2x_1-1\,,\,\frac12(x_2+1)) & \frac12\le x_1\le 1 \end{cases}.
\]
We identify the first component of $f$ with a map $\bar f:[0,1]\to[0,1]$
of the form in~\eqref{eq-LSV} with $\gamma\ge1$.
Recall from the introduction that $\bar f$ has a unique (up to scaling) absolutely continuous invariant measure $\bar\mu$ and this is an infinite measure.
Let $\bar Y=[\frac12,1]$ and 
define $\varphi:\bar Y\to\Z^+$, $\bar F=\bar f^{\varphi}:\bar Y\to \bar Y$ to be the first return
time and first return map.
Then $\varphi$ satisfies conditions~\eqref{eq-varphi} and~\eqref{eq-smooth1} relative to $\bar\mu$ with $\beta=1/\gamma$ and $\ell$ asymptotically constant.
Moreover $\bar F$ is a Gibbs-Markov map.

Turning to the full map $f$, we set
$Y=[\frac12,1]\times[0,1]$
and define $\varphi:Y\to\Z^+$, $F=f^\varphi:Y\to Y$ to be the first return
time and first return map.  Clearly $\varphi(y_1,y_2)=\varphi(y_1)$ and
$F$ has the form $F(y_1,y_2)=(\bar F(y_1),G(y_1,y_2))$ where 
$|\partial G/\partial y_2|_\infty\le \frac12$.

We now verify the hypotheses from Section~\ref{sec-NUH}.
Let $\{Z_j\}=\{a\times[0,1]:a\in\alpha\}$ where $\alpha$ is the partition of $\bar Y$ for the Gibbs-Markov map $\bar F$ and set $\varphi_j=\varphi|_{Z_j}$.
Define the family of stable disks $\{W^s\}=\{\{y_1\}\times[0,1]:y_1\in \bar Y\}$
and the family of unstable disks $\{W^u\}=\{\bar Y\times\{x_2\}:x_2\in [0,1]\}$.
Then (P1) and (P2)(i) are immediate with $\gamma_0=\frac12$.
Also hypotheses (P2)(ii) and (P2)(iii) are standard for $\bar f$ and $\bar F$, and 
are clearly inherited by $f$ and $F$.
Hypotheses (P3) and (P4) are immediate.
Moreover, hypothesis (E) holds (also with $\gamma_0=\frac12$). 
\end{examp}

\begin{examp}
\label{ex-LT}
More generally, we consider invertible maps $f:X\to X$ of the form
$f(x_1,x_2)=(\bar f(x_1),g(x_1,x_2))$ with $\bar f$, $\bar Y$, $\varphi$ and $\bar F$
as before, and it is assumed that $g:X\to[0,1]$ is $C^2$.  Then $\varphi$ and $Y$ are also unchanged and $F=f^\varphi:Y\to Y$ has the form
$F(y_1,y_2)=(\bar F(y_1),G(y_1,y_2))$.  We suppose that $|\partial G/\partial y_2|_\infty\le\gamma_0$
where $\gamma_0\in(0,1)$.
The stable disks $\{W^s\}$ and the partition $\{Z_j\}$ are unchanged from
Example~\ref{ex-trivial}, and the existence of unstable disks $\{W^u\}$ is a standard consequence of the uniform hyperbolicity of $F|_{Z_j}$ for each $j$.
Hypotheses (P1)--(P4) follow.
If we assume further that $|\partial g/\partial x_2|_\infty\le\gamma_0$
then hypothesis (E) holds.
\end{examp}

\begin{examp} \label{ex-SW}
A rather different set of examples can be constructed along the lines of the Smale-Williams solenoids~\cite{Smale67,Williams67}.  First modify the map $\bar f$ near $\frac12$ so that
$\bar f$ is $C^2$ on $(0,1)$.  Let $X=[0,1]\times \D$ where
$\D$ is the closed unit disk in $\R^{n-1}$ and define
$f_0(x_1,x_2)=(\bar f(x_1),g(x_2))$ for $(x_1,x_2)\in [0,1]\times\D$
where $g:\D\to\D$ is $C^2$ and $|Dg|_\infty<1$.
Now perturb to obtain a $C^2$ embedding $f:X\to X$ such that $f=f_0$ on $[0,\frac12]\times\D$.  
Note that $f$ has a unique neutral fixed point inside
$\{0\}\times\D$.

Let $Y=[\frac12,1]\times\D$ and define the first return time
$\varphi:Y\to\Z^+$ and first return map $F=f^\varphi:Y\to Y$.
The combinatorics in this example are almost identical to those in the previous ones, the only difference being in the first iterate of $f$, so it is easy to check the conditions~\eqref{eq-varphi} and~\eqref{eq-smooth1}.  Also, $F$ is uniformly hyperbolic if $f$ is close enough to $f_0$ and conditions (P1)--(P3) as well
as (E) are easily verified.  
Moreover, (P4) is satisfied provided the perturbation $f$ is chosen so that
$\bar F$ is Markov.
\end{examp}

\begin{examp} \label{Ex-E}
Finally, we mention a class of examples for which hypothesis (E) fails.
Proceed exactly as in Example~\ref{ex-SW} except that $g:\D\to\D$ is replaced 
by a $C^2$ map $g:X\to\D$ with derivative $Jg=D_{x_2}g$ in the $\D$ directions
such that
$|Jg|_\infty\le1$, $|Jg|_{[\frac12,1]\times\D}|_\infty < 1$,
but $Jg(0,x_2)\equiv I$.
This means that the invariant set $\{0\}\times\D$ is neutral in all directions.  However, writing $F=(F_1,F_2)$, we have that
 $|D_{x_2}F_2|_\infty\le |Jg|_{[\frac12,1]\times\D}|_\infty <1$ and hypotheses (P1)--(P4) hold as before.
\end{examp}

\begin{rmk} \label{rmk-LT}
Independently, using different techniques the bypass the quotienting step altogether, Liverani and Terhesiu~\cite{LiveraniTerhesiusub} have obtained similar results to the ones presented in this paper.
They work with distributional function spaces of the type pioneered by Liverani and coworkers starting with~\cite{BlankKellerLiverani}.

The class of examples in Example~\ref{ex-LT} is essentially the same as those treated by Liverani and Terhesiu~\cite[Example~1]{LiveraniTerhesiusub}.
Currently the function spaces 
in~\cite{LiveraniTerhesiusub} are restricted to situations where there
is a global smooth stable foliation.  In contrast, the method presented in this paper is not dependent on the smoothness of the stable foliation.  In particular, our methods apply to sufficiently small smooth perturbations, preserving the Markov structure, of the maps
in Example~\ref{ex-LT}, and their higher dimensional analogues.  For such maps, generally the stable foliations have only H\"older regularity (even though the individual leaves are smooth).

Similarly, the methods in~\cite{LiveraniTerhesiusub} currently do not   generally apply to Examples~\ref{ex-SW}
and~\ref{Ex-E}.  On the other hand, the non-Markov examples of
\cite[Example~2]{LiveraniTerhesiusub} are not covered by the method
presented in this paper.

It will be interesting to see whether the
the methods in~\cite{LiveraniTerhesiusub} can be developed to prove
mixing rates and higher order asymptotics in situations such as in
 Example~\ref{Ex-E} when there is not a global smooth stable foliation.
In many important situations such as billiards, hyperbolicity breaks down due to long periods where there is no contraction or expansion, so that (P1)--(P4) hold but not (E).  
Such examples are studied for instance in~\cite{ArbietoMarkarianPacificoaSoares12}.
\end{rmk}

\paragraph{Acknowledgements}
This research was supported in part by the CNRS during a three month appointment as chercheur en math\'ematique at CIRM, Marseille, Autumn 2011 where this research was begun.
The research was continued at the University of Surrey and completed at the University of Warwick, partially supported by
the European Advanced Grant {\em StochExtHomog} (ERC AdG 320977).
I am very grateful to Dalia Terhesiu for numerous helpful and encouraging discussions at various stages of the work.

\end{document}